\documentclass[a4paper,twoside,english]{amsart}
\usepackage[T1]{fontenc}
\usepackage[latin9]{inputenc}
\usepackage{babel}
\usepackage{verbatim}
\usepackage{amsthm}
\usepackage{amsmath}
\usepackage{amssymb}
\usepackage[unicode=true,
 bookmarks=true,bookmarksnumbered=true,bookmarksopen=false,
 breaklinks=false,pdfborder={0 0 1},backref=false,colorlinks=false]
 {hyperref}
\hypersetup{
 pdfauthor={V. Mantova and U. Zannier}}

\makeatletter

\pdfpageheight\paperheight 
\pdfpagewidth\paperwidth

\theoremstyle{plain}
\newtheorem{thm}{\protect\theoremname}[section]
 \newcommand\thmsname{\protect\theoremname}
 \newcommand\nm@thmtype{theorem}
 \theoremstyle{plain}

  \theoremstyle{remark}
  \newtheorem{rem}[thm]{\protect\remarkname}
  \theoremstyle{definition}
  \newtheorem*{example*}{\protect\examplename}
  \theoremstyle{definition}
  
  \theoremstyle{plain}
  
  \theoremstyle{plain}
  \newtheorem{prop}[thm]{\protect\propositionname}
  \theoremstyle{plain}

\makeatletter
  \theoremstyle{definition}
  \newtheorem{my@rem}[thm]{Remark}
  \renewenvironment{rem}{\begin{my@rem}}{\end{my@rem}}
\makeatother

\makeatother

  \providecommand{\examplename}{Example}
  \providecommand{\lemmaname}{Lemma}
  \providecommand{\propositionname}{Proposition}
  \providecommand{\remarkname}{Remark}
  \providecommand{\theoremname}{Theorem}
\providecommand{\theoremname}{Theorem}
 \providecommand{\corollaryname}{Corollary}
\usepackage[left=3.1cm,right=3.5cm,top=3cm,bottom=3cm]{geometry}

\usepackage{mathtools}

\def\Q{{\Bbb Q}}
\def\R{{\Bbb R}}

\def\L{{\mathcal L}}

\def\Z{{\Bbb Z}}

\def\H{{\mathcal H}}

\def\s{{\mathcal S}}

\def\C{{\Bbb C}}

\def\CVD{{\hfill\hfil{\lower 2pt\hbox{\vrule\vbox to 7pt
{\hrule width  5pt\varphifill\hrule}\varphirule}}}\par}

\begin{document} 

\title{ Integral distances  from (two) given lattice points}

%%% Integral distances among lattice points
%%% Integral distances among three lattice points
%%%  Integral distances  from (two) lattice points

\author{Umberto  Zannier}
%\date{December, 2016}
\maketitle

\date{\today}

\centerline{\sc Abstract}

\medskip

{\it We completely characterize pairs of lattice points $P_1\neq P_2$ in the plane with the property that there are infinitely many 
lattice points $Q$ whose distance from both $P_1$ and $P_2$ is integral.

 In particular  we show that it suffices that  $P_2-P_1\neq (\pm 1,\pm 2), (\pm 2,\pm 1)$, and  we show that $|P_1-P_2|>\sqrt{20}$ suffices for having infinitely many  such $Q$ outside any finite union of lines. 
 
 We use only elementary arguments, the crucial ingredient being a theorem of Gauss which does not appear  to be  often applied. We  further  include related   remarks (and open questions), also  for distances from  an  arbitrary prescribed  finite set of  lattice points.} % $P_1,\ldots ,P_r$. }

 \bigskip

\section{Introduction}

In this short elementary article we shall be concerned with {\it integral distances} from {\it given  lattice points}, which here means   points in  the integral lattice $\Z^2$,  i.e.,   points in the plane $\R^2$ having integer coordinates. We suppose  that certain lattice points are  given in advance, and let another lattice point vary, asking that  it has integral distance from each of the given points. 
 
 \medskip
 
This kind of  issue  has been often   considered in various shapes. Of course, the  natural  case %problem  to  describe the 
of lattice points  having integral distance from   {\tt one given}  lattice point  boils down  to %the description of
 the {\it Pythagorean triples}, i.e., the triples $(a,b,c)$ of integers such that $a^2+b^2=c^2$. Indeed, the given point  may be taken as the origin $O$ and  letting the variable  point have coordinates $a,b$ and  distance $c$  from $O$,  Pythagoras Theorem 
yields  the said equation. 
%%%%%%%%%%%%%%%%%%%%%%%%%%%%%%%%%%%%%%%%%%%%%%%%%%%%%%%%%%%%%%%%%%%%%%%%%%%%%%%%%%%%%%%%%%%%%%%%%%%%%%%%%%%%%%%%%%%%%%%%%%%%%%%%%%%%%%%%%%%%%%%%%%%%%%%%%%%%%%%%%%%%%%%%%%%%%%%%%%%%%%%%%%%%%%%%%%%%%%%%%%%%%%%%%%%%%%%%%%%%%%%%%%%%%%%%%%%%%%%%%%%%%%%%%%%%%%%%%%%%%%%%%%%%%%%%%%%%%%%%%%%%%%%%%%%%%%%%%%%%%%%%%%%%%%%%%%%%%%%%%%%%%%%%%%%%%%%
\begin{comment}
a  solution $(a,b,c)$ of the diophantine equation
 \begin{equation}\label{E.pit}
 x^2+y^2=z^2,\qquad x,y,z\in\Z.
 \end{equation}
 \end{comment}
 %%%%%%%%%%%%%%%%%%%%%%%%%%%%%%%%%%%%%%%%%%%%%%%%%%%%%%%%%%%%%%%%%%%%%%%%%%%%%%%%%%%%%%%%%%%%%%%%%%%%%%%%%%%%%%%%%%%%%%%%%%%%%%%%%%%%%%%%%%%%%%%%%%%%%%%%%%%%%%%%%%%%%%%%%%%%%%%%%%%%%%%%%%%%%%%%%%%%%%%%%%%%%%%%%%%%%%%%%%%%%%%%%%%%%%%%%%%%%%%%%%%%%%%%%%%%%%%%%%%%%%%%%%%%%%%%%%%%%%%%%%%%%%%%%%%%%%%%%%%%%%%%%%%%%%%%%%%%%%%%%%%%%%%%%%%%%%%%%%%%%%%%%%%%%%%%%%%%%%%%%%%%%%%%%%%%%%%%%%%%%%%%%%%%%%%%%%%%%%%%%%

 Needless to say, these triples, of which the simplest nontrivial   is $(3,4,5)$,  are more than well known, having a very   ancient origin, the oldest   record coming from the  Babylonian tablet  `PLIMPTON 322', dated about 1800 B.C..  
A general   parametrization of these triples  % the solutions of \eqref{E.pit} 
 also  goes back to long ago (at latest to Euclid, see Weil's book \cite{W} for an accurate historical account)  and is very well known and easily described: if $(a,b,c)$ is a solution, on switching  if necessary $a,b$ we have that $a,b,c$ have the shapes given by $a=d\cdot 2pq,\quad b=d\cdot (p^2-q^2),\quad c=d\cdot (p^2+q^2)$,
% \begin{equation*}
 %a=d\cdot 2pq,\qquad b=d\cdot (p^2-q^2),\qquad c=d\cdot (p^2+q^2),
% \end{equation*}
 where $d,p,q$ are suitable integers, such that  $p,q$ are coprime and have opposite parity.  These triples were considered and used by classical writers, like Diophantus, and later Fermat, who most probably was inspired by them to formulate his `Last Theorem'. %\footnote{Of course this is the assertion that for $n>2$ the diophantine equation $a^n+b^n=c^n$ has only the trivial  solutions with $abc=0$,  proved only in the 90s by A. Wiles after efforts by many  others.}
 
 \medskip
 
Now, in analogy with the above,  it seems challenging  to ask the following:

\medskip

{\bf Question}: {\it  What can be said about the  lattice points $Q$ having integral distance from each out of  ({\it two} or more)  {\tt given} (distinct) lattice points $P_1,P_2,\ldots ,P_r$ ?} 

\medskip

For instance, in the same direction of the usual  diophantine queries, one can ask: 

\medskip

\centerline{\it When do these lattice points $Q$  make up an  infinite set ?}

\medskip

 This context reminds of the  famous   Anning-Erd\"os theorem, asserting that no infinite set of points in the plane can have {\it all} mutual distances  integral, unless all the points are collinear  (see the paper \cite{AE} by both Anning and Erd\"os and see also Erd\"os' article \cite{E} for a simplification). However the present issue, though certainly related to this,  is different,  since we are {\it  fixing}  some points in advance, and we only look  at the distances of $Q$ from each of  the given points (without conditions on the distances among them). 

\medskip

 When e.g.   $r=2$ {\it and} the distance $|P_1-P_2|$ is  integral, it is easy to see that  for our points $Q$  the triangle $P_1P_2Q$ is {\it Heronian},  % `up to a factor $2$',
 i.e., has integral sides and %half-
 integral area.\footnote{Indeed the area is  a difference between the area of an integral rectangle and some right-angled triangles with integer sides; each of them has integer area since one leg at least must have even length.}  Conversely, it has been proved by P. Yiu \cite{Y}  that any Heronian  triangle %in $\R^2$  
  is congruent % again up to a factor $2$,  
 to a {\it lattice triangle}, i.e., to  a triangle having all  vertices at lattice points. And J. Carlson proved that there are infinitely many Heronian triangles with a side of given integral length (see  \cite{C}, Thm. 2). 
Hence our issue is not unrelated to Heronian  triangles as well; these  have been studied  and parametrized (see e.g. \cite{B} and \cite{C}). However, again there are also several  differences. In fact: 
(i) we are not assuming that any of the distances among the $P_i$ is integral. 
(ii) We are thinking of $P_1,P_2,\ldots ,P_r$, as being given in advance  (so that our  attention   is on a point - $Q$ - rather than a triangle). 
(iii) Even assuming $r=2$ and $|P_1-P_2|$  integral, although a Heronian triangle congruent with $P_1P_2Q$ could be imbedded as a lattice triangle (by the cited paper \cite{Y}), {\it a priori} the side $P_1P_2$ could be different from the corresponding side of the embedded Heronian triangle (namely, it could  have a different slope still being of the same length). 

\medskip

So the present  problem falls aside Heronian triangles  in several aspects, and in fact we have no knowledge of it in the existing  literature (which seemed to us somewhat surprising).

\medskip

To go ahead, it will be convenient to introduce a minimum notation.

First, we shall denote by  $(\cdot ,\cdot)$ the usual scalar product in $\R^2$, and  by $|\cdot |$ the associated distance, so that, as above,  $|Q-P|$ is   the (usual euclidean) distance between $P$ and $Q$. 

Also, for $P_1,\ldots ,P_r\in\Z^2$   distinct lattice points, we set
\begin{equation}\label{E.S}
\s=\s(P_1,\ldots ,P_r)=\{Q\in\Z^2: |Q-P_i|\in\Z\  \hbox{for $i=1,\ldots ,r$}\}.
\end{equation}

%%%%%%%%%%%%%%%%%%%%%%%%%%%%%%%%%%%%%%%%%%%%%%%%%%%%%%%%%%%%%%%%%%%%%%%%%%%%%%%%%%%%%%%%%%%%%%%%%%%%
%%%%%%%%%%%%%%%%%%%%%%%%%%%%%%%%%%%%%%%%%%%%%%%%%%
%%%%%%%%%%%%%%%%%%%%%%%%%%%%%%%%%%%%%%%%%%%%%%%%%%
%%%%%%%%%%%%%%%%%%%%%%%%%%%%%%%%%%%%%%%%%%%%%%%%%%
%%%%%%%%%%%%%%%%%%%%%%%%%%%%%%%%%%%%%%%%%%%%%%%%%%
%%%%%%%%%%%%%%%%%%%%%%%%%%%%%%%%%%%%%%%%%%%%%%%%%%
%%%%%%%%%%%%%%%%%%%%%%%%%%%%%%%%%%%%%%%%%%%%%%%%%%

\begin{comment}
More generally, for a subring $A$ of $\C$, we set
\begin{equation*}\label{E.S2}
\s_A=\s(_AP_1,\ldots ,P_r)=\{Q\in A^2: |Q-P_i|\in\Z\  \hbox{for $i=1,\ldots ,r$}\},
\end{equation*}
so $\s=\s_\Z$, the case we are really thinking about.  Observe that $\s\subset\s_A$ as soon as $1\in A$.
\end{comment}
%%%%%%%%%%%%%%%%%%%%%%%%%%%%%%%%%%%%%%%%%%%%%%%%%%
%%%%%%%%%%%%%%%%%%%%%%%%%%%%%%%%%%%%%%%%%%%%%%%%%%
%%%%%%%%%%%%%%%%%%%%%%%%%%%%%%%%%%%%%%%%%%%%%%%%%%
%%%%%%%%%%%%%%%%%%%%%%%%%%%%%%%%%%%%%%%%%%%%%%%%%%
%%%%%%%%%%%%%%%%%%%%%%%%%%%%%%%%%%%%%%%%%%%%%%%%%%
%%%%%%%%%%%%%%%%%%%%%%%%%%%%%%%%%%%%%%%%%%%%%%%%%%
%%%%%%%%%%%%%%%%%%%%%%%%%%%%%%%%%%%%%%%%%%%%%%%%%%

\medskip

We are mainly interested in understanding  how `large' $\s$ is, and especially in saying when $\s$ is infinite. So, for $r=1$ we have just  the Pythagorean triples, completely described by the formulas recalled  above. In particular their set is {\it Zariski-dense} in the cone $x^2+y^2=z^2$ in $3$-space, i.e., there is no algebraic curve inside the cone containing all the integer points. Hence there is no algebraic curve in the plane containing $\s(P_1)$ (even  restricting the points to have coprime coordinates, after taking $P_1=O$).  

To discuss this issue for $r>1$,  let us start with an easy assertion, which in fact is essentially well known and is  inserted here only for completeness (with a proof based on the same principles as in the quoted papers on the Anning-Erd\"os problem). We formulate it as a Proposition:

\begin{prop} \label{P.r3} The set $\s(P_1,P_2)$ is contained in a finite union of hyperbolas plus the (orthogonal)  lines  $P_1P_2$ and  the line of points equidistant from $P_1,P_2$. 

For  $r\ge 3$, the set  $\s$ is finite and effectively computable unless all  the $P_i$  are collinear  and have mutual integral distances, and then  all  but finitely many points in $\s$ lie on the  line $P_1P_2$.
\end{prop}

\begin{small}
\begin{rem} The (easy) proof actually yields the same assertions even letting $Q$ run through  $\R^2$, i.e., dropping the request   that the coordinates of $Q$ lie in $\Z$ (but still requiring that the {\it distances}  from the $P_i$ are integers).
\end{rem}
\end{small}

%\medskip

In particular, the second part of the proposition  gives back  the Anning-Erd\"os theorem quoted above. Interesting questions arise if we ask for explicit bounds for the cardinality $\# \s$, which we shall briefly comment on at the end.

This proposition says in particular that the crucial case for the infinitude of $\s$ occurs when  $r=2$, which was in fact  our  motivation for this note. We have the following remark: 

\begin{thm}\label{T.r2} 
For every finite union $\L$ of lines, the set $\s(P_1,P_2)- \L$ is infinite, unless the point  $P:=P_2-P_1$, after possible sign changes and switching of its coordinates,  belongs to the following list (where we replace by translation $P_1,P_2$, resp. by $O,P=P_2-P_1$).

(i) $P=(0,1)$: now  $\s$ consists of the integer points on the $y$-axis. 
 
(ii)  $P=(1,1)$: now  $\s$ is infinite and contained in the line $x+y=1$.  
  
(iii) $P=(0,2)$: now  $\s$ consists of the integer  points on the $y$-axis.
  
(iv) $P=(1,2)$: now  $\s=\{(1,0), (0,2)\}$. 
  
 (v) $P=(2,2)$: now $\s$ is infinite and contained in the line $x+y=2$.  
 
 (vi)  $P=(2,4)$: now  $\s$ is  the union of an infinite set  contained in the line $x+2y=5$ and the set $\{(0,2), (4,0), (-1,0), (3,4)\}$. 
\end{thm}

In particular, the first conclusion applies if $|P_1-P_2|>\sqrt{20}$; and $\s$   turns out anyway  to be  infinite unless $|P_1-P_2|=\sqrt 5$, which amounts to $P_1-P_2\in\{(\pm 1,\pm2), (\pm 2,\pm 1)\}$. 

Our proof of Theorem \ref{T.r2} will be short and entirely elementary,   based on   the classical theory of Pell Equation.  The main point is the use of a theorem of Gauss which,  to my knowledge,  is only seldom applied, and seems not  to be   as   well known as one would expect. 
The arguments  will also give supplementary information on the distribution of $\s(P_1,P_2)$, on which we shall briefly comment at the end.

\medskip

{\bf Acknowledgements}. I thank  David Masser for helpful  comments and Amos Turchet for a careful reading and the suggestion of references. 

\section{Proofs}

\begin{proof}[Proof of Proposition \ref{P.r3}] Set $d_i:=|Q-P_i|$, so the $d_i$ are integers $\ge 0$. Since we have $d_i\le d_j+|P_i-P_j|$,  and since the $P_i$ are given, the differences  $k_{ij}:=d_i-d_j$ are integers bounded in absolute value: $|k_{ij}|\le |P_i-P_j|$;  hence they can assume only finitely many values for varying $Q\in \s$. This is the bulk of the matter,  and we may correspondingly partition $\s$ into  finitely many sets.  Suppose then to fix the $k_{ij}$, and denote by $\s'$ the corresponding subset of $\s$. For given $i\neq j$, the equation  (for $Q$) given by 
\begin{equation}\label{E.hyp}
|Q-P_i|=|Q-P_j|+k_{ij},\qquad Q\in \R^2,
\end{equation}
defines (if $|k_{ij}|\le |P_i-P_j|$)a branch of a  hyperbola, possibly degenerating to a (half)  line, as is known from high school (in fact essentially by definition).       Let us inspect this. 

To simplify notation, set  $P_j=O$, $P_i=P$, $k_{ij}=k$. Then squaring \eqref{E.hyp}  \footnote{Note that we may gain extra solutions after squaring.}  and noting that $|Q-P|^2=|Q|^2-2(Q,P)+|P|^2$ easily yields   
\begin{equation}\label{E.hyp2}
-2(Q,P)+|P|^2-k^2=2k|Q|.
\end{equation}

This also easily says when the equation defines a (half) line. Indeed, we may assume by a rotation  that the line in the $xy$-plane  has equation $y=c$, leading, if $P=(a,b)$, $Q=(x,c)$, to $-2ax-2bc+a^2+b^2-k^2=2k\sqrt{x^2+c^2}$.  Suppose that  this holds for at least three values of  $x\in\R$; then it must hold  identically. In this case,    either $k=0$,   yielding  that  $Q$ is equidistant from $P_i,P_j$ (which holds when $Q$ runs through  a whole line),   or  $x^2+c^2$ is  the square of a linear polynomial in $x$. But this holds if and only if 
$c=0$, in which case  $k^2=a^2+b^2=|P|^2$, $-a=k$,  so $b=0$ and $O,P,Q$ are collinear. Also,  $Q$ runs through a half of   the line $P_iP_j$,  which half being determined by the sign of $k$ and  by  the property that it does not intersect the interior of the segment between $P_i$ and $P_j$.   

This proves the first claim. Note also  that we may check effectively the totality of  these conditions, once the $P_i$ are effectively given (because then  the $k_{ij}$ vary in only  finitely many ways which can be enumerated). 

\medskip

We may now suppose  that  $r\ge 3$, and,  to simplify notation,    that one of the $P_i$ is $O$, labelling two other ones by $P,P'$ and the respective constants by $k,k'$. Equations \eqref{E.hyp2} for $P,P'$ yield that 
 $(Q,kP'-k'P)$ is constant for $Q\in \s'$, so either $kP'=k'P$, or all relevant $Q$ lie on a single line.  The latter case has been discussed above: either we obtain  at most  two points  $Q$, or  $O,P,P'$ are collinear and  all $Q\in \s'$  lie on the corresponding  line.  Let us then  assume that $kP'=k'P$.  If $k=k'=0$ then $Q$ is equidistant from $O,P,P'$ which yields at most a single  point $Q$.  Otherwise, let $k\neq 0$, so $P'=(k'/k)P$ and  we obtain again that   $O,P,P'$    are collinear and $|P'|^2=(k'/k)^2|P|^2$. Also, $k'\neq 0$ (since $P'\neq O$) and $k'\neq k$ (since $P\neq P'$). Plugging this into \eqref{E.hyp2} for $P$ and $P'$ we obtain easily $k'(k-k')|P|^2=k^2k'(k-k')$, whence $|P|=|k|$, $|P'|=|k'|$. But then $Q$ is collinear with $O,P,P'$. 
 
 Repeating the argument for each triple of points and all  values of the $k_{ij}$, we obtain that either $\s$ is finite and computable, or all the $P_i$ are collinear and up to  finitely many computable exceptions each $Q\in \s$   lies on the corresponding line. Plainly in this case either this last set  is empty or all  the   $|P_i-P_j|$ are integral,  completing the proof.
 \end{proof} 

\begin{small}
As remarked above,  the arguments never use that the relevant  $Q$ are lattice points, only that the distances from the $P_i$ are integers, so the conclusions hold  for points $Q\in\R^2$ as well.
\end{small}

\medskip

\begin{proof}[Proof of Theorem \ref{T.r2}]
One of the points $P_1,P_2$ may be  supposed to be the origin $O$, and let us denote by $P=:(a,b)\neq O$ the  other point. As in the statement, by an integral orthogonal transformation (i.e., up to sign changes and switch of coordinates) we may suppose without loss that $b\ge a\ge 0$.

Let us first assume  that $P=:(a,b)$ is not in the list of exceptions,  namely that $P\neq (0,1), (1,1),  (0,2), (1,2), (2,2), (2,4)$.    

For a point  $Q=(x,y)\in \s$, let $z:=|Q|$, $z-k:=|Q-P|$, so $x,y,z,k$ are integers with $z\ge 0$, $-|P|\le |Q|-|Q-P|=k\le |P|$. We have thus the equations
\begin{equation}\label{E.xyz}
x^2+y^2=z^2,\qquad (x-a)^2+(y-b)^2=(z-k)^2.
\end{equation}

\begin{small}
We note in passing that for $a^2+b^2\neq k^2$ the Pythagorean triples provided by any solution of \eqref{E.xyz} are `essentially' primitive. In fact,  both $\gcd(x,y,z)$ and $\gcd(x-a,y-b,z-k)$ divide  $a^2+b^2-k^2$.
\end{small}

\medskip

Conversely, for a given integer $k$, an integral solution $(x,y,z)$ of \eqref{E.xyz} yields a point $Q=(x,y)\in \s$ (even though $z$ need not be equal to $|Q|$).  

Note also that there can be integral solutions only if $k\equiv a+b\pmod 2$. In general in the sequel we shall work  with integral values of $k$ satisfying this condition and also such that $k\neq 0$ and $k^2<a^2+b^2$. 

The  equations \eqref{E.xyz} define a curve in affine $3$-space.\footnote{It turns out that this curve is irreducible unless  $\delta:=a^2+b^2-k^2=0$:  see equations  \eqref{E.z}  and \eqref{E.sost2} below. Also note that for $\delta=0$ the left-hand side of \eqref{E.sost2}  equals $-(bx-ay)^2$. Equations   \eqref{E.z}  and \eqref{E.sost2}  also  show  that the ideal generated by equations \eqref{E.xyz} is not reduced when $k\delta=0$.}  We want to find its projection on the $xy$-plane, i.e. to eliminate $z$. Subtracting the second equation from the first we obtain
\begin{equation}\label{E.z}
2kz=2ax+2by-\delta,
\end{equation}
where we have put
\begin{equation}\label{E.delta}
 \delta=a^2+b^2-k^2.
\end{equation}
Now, multiplying the first of \eqref{E.xyz} by $4k^2$ and using \eqref{E.z}, we get
\begin{equation}\label{E.sost2}
(2ax+2by-\delta)^2-4k^2(x^2+y^2)=0,
\end{equation}
which may be written in the shape 
\begin{equation}\label{E.sost}
4(a^2-k^2)x^2+8abxy+4(b^2-k^2)y^2-4a\delta x-4b\delta y+\delta^2=0.
\end{equation}
\medskip

 For  {\it given}  integers $a,b,k$  not all zero, equation \eqref{E.sost} represents an affine conic, denoted $\H$, %, and the discussion in the proof of Proposition \ref{P.r3} proves that 
 and any $Q=(x,y)\in \s$ leads to an integral point on it.
 Conversely, if $x,y$ is an integral solution (for given integers $a,b,k$)  then  we have $(2ax+2by-\delta)^2=4k^2(x^2+y^2)$, hence $2k$ divides $2ax+2by-\delta$. So if we assume $k\neq 0$,  this gives an integral value for $z$ as defined by \eqref{E.z} and an integral solution of the system \eqref{E.xyz}. Hence we obtain a point $Q=(x,y)\in \s$. 
  \medskip

 The shape \eqref{E.sost2} also easily shows that for $k\delta\neq 0$ the affine conic $\H$  is irreducible (even over $\C$). Indeed, let $L$ be a line defined by a hypothetical  linear factor. Then, since $k\neq 0$, the polynomial   $x^2+y^2=(x+iy)(x-iy)$ restricted to $L$ would be a perfect square (by \eqref{E.sost2}), whence either one of the  factors $x\pm iy$ would define $L$  or the restrictions of  both $x\pm iy$ to $L$ would be equal up to a constant factor, and  in any of these cases  the line would pass through the origin. But then $\delta=0$ by \eqref{E.sost2}. 
 
 \medskip
 
The homogenous binary  form of degree $2$ in the equation \eqref{E.sost} defining $\H$ yields its points at infinity;  the    discriminant  of this form  is  easily calculated as $64\left(a^2b^2-(a^2-k^2)(b^2-k^2)\right)=64k^2\delta$.  If  we assume  $k\neq 0$ and $\delta>0$, the conic is a hyperbola (i.e. it has two {\it real} points at infinity). Now, the theory of Pell Equation tells us many things about  integer points on hyperbolas. In particular, Gauss deduced from this  theory  the following theorem:
 
 \medskip
 
 {\sc Gauss Theorem}.  {\it  Let a quadratic polynomial in $x,y$ with integer coefficients define  an (absolutely)  irreducible (affine)  hyperbola and  assume that the discriminant of its quadratic homogeneous part is not a perfect square (i.e. the points at infinity are not defined over $\Q$).
 
 Then if there is one integer point on $\H$, there are infinitely many ones.}
 
 \medskip

An equivalent   statement indeed appears  in Gauss' {\it Disquisitiones Arithmeticae} at art. 216, $3^o$ (see for instance   the translation \cite{G}). See, e.g., L.J.  Mordell's book \cite{M}, Thm. 2, p. 57 for a more modern presentation  of  a proof of this theorem (or see the writer's book \cite{Z}, p. 21), and see the paper  \cite{ABP} for a  generalization.

  \begin{small}
  Due to the elementary nature of this article, and for the  reader's convenience we resume the proof-principle, which  is simple: by `completing the square' one writes the equation in the shape $X^2-DY^2=C$, where $C,D$ are nonzero integers, with $D$ the said discriminant, and where $X,Y$ are  polynomials  in $x,y$ of degree $1$ with integer coefficients. An integer solution $X_0,Y_0$ of the new equation gives back {\it  integers}  $x_0,y_0$ precisely if  $X_0,Y_0$ satisfy certain congruences relative to a fixed modulus $M\neq 0$ depending on the said polynomials. These congruences are  unaffected if we `compose' a solution with a solution of the Pell Equation  $T^2-DU^2=1$  such that $T\equiv 1, U\equiv 0\pmod M$ (this composition corresponds to multiplication in $\Z[\sqrt D]$).  Then, since the Pell Equation $T^2-DM^2V^2=1$   always has infinitely many integer solutions (as expected by Fermat and proved by Lagrange), we obtain an infinity of solutions of our equation  by composition from any given solution. 
 \end{small}
 
 \medskip
 
 Turning back to our context, to prove the theorem it then suffices, for given $(a,b)$ not in the said list,  to produce an integer point on some hyperbola as above, such that $k\neq 0$, and such that $\delta$ is a positive integer not a perfect square.  (Indeed, once we find an infinity of points in $\s$ lying on an irreducible hyperbola, omitting those lying on the finite union $\L$ of lines still leaves us with an infinite set.) 
 
 Let us try with points $Q=(x,b)$. The equations \eqref{E.xyz} become $x^2+b^2=z^2$ and $(x-a)^2=(z-k)^2$. The latter is satisfied if we put $z=x+k-a$. %, where $e=\pm 1$. 
 Substituting into the former we obtain $x^2+b^2=x^2+2(k-a)x+(k-a)^2$, i.e. 
 \begin{equation*}
2(k-a)x= b^2-(k-a)^2.
 \end{equation*}
 To have an integer value for $x$ amounts to $b^2-(k-a)^2$ being multiple of $2(k-a)$. 
 
 \medskip
 
- Suppose first that $b$ is odd. 

If $a\neq 1$ let us choose $k=a-1\neq 0$. % where $e=\pm 1$. 
Then $k-a=-1$ and the divisibility condition is verified. Also, $\delta=a^2+b^2-(a-1)^2=b^2+2a-1$.  Suppose this equals a perfect square $r^2$, $r\ge 0$.  If $a=0$ this entails $b=1$,  which we are excluding. If $a>0$ then $\delta>b^2>0$ so $r\ge b+1$, which is contradictory since $b\ge a$. 

 If $a=1$ let us put $k=2$, so $2(k-a)=2$ which divides $b^2-1$. Then $\delta=b^2-3$. If  $b\ge 3$ this  is positive and cannot be a square, so we are done. This leaves us with  the case $a=b=1$, indeed in the list of exceptions.

 - Things  are similar if $b$ is even. 
 
 Now, if $a\neq 2$ we put  $k=a-2\neq 0$, and again $2(k-a)=-4$ divides $b^2-(k-a)^2$. We have $\delta=b^2+4a-4$.  This is $>0$ unless $a=0,b=2$,  which is in the list of  exceptions. Otherwise, if $a=0$ this cannot be a square (since $b^2-(b-2)^2=4b-4>4$ if $b>2$). If $a>1$ then $\delta>b^2$ and, being even, $\delta$  must be $\ge (b+2)^2=b^2+4b+4>b^2+4a-4$, a contradiction. If $a=1$ or $a=2$ then we put $k=a+2$, again the said divisibility being verified. We have $\delta=b^2-4a-4$, which is $=b^2-8$ or $b^2-12$ in the two cases. This is $>0$ unless $a=1,b=2$ or $a=b=2$, which are exceptional. If $\delta>0$ is a perfect square then $(b/2)^2-2$ or $(b/2)^2-3$ is a perfect square as well. The first case is impossible mod $4$ and the second case entails $b=4$, and  we get the point  $(a,b)=(2,4)$ again exceptional. 
 
 \medskip
 
  This completes the proof of the first part of the theorem.
 
 \medskip
 
 To prove the second  part we again can work with $O,P=P_2-P_1$ in place of $P_1,P_2$, and with the assumption $b\ge a\ge 0$;  we can essentially reverse the above arguments. Let us be explicit. Recall that if we have solutions corresponding to an integer $k$ as above, then $|k|\le |P|$ and $k\equiv a+b\pmod 2$.  
 
$\bullet$  (i)   If $P=(0,1)$ we must have $k=\pm 1$, so $\delta=0$. The equation \eqref{E.sost2} yields $x=0$, so indeed $\s$ consists of the integer points on the $y$-axis. 
 
$\bullet$  (ii)   The case $P=(1,1)$ forces $k=0$, $\delta=2$, and we get the line $L:x+y=1$, of points equidistant from $O,P$.   Now $k=0$ so the above procedure must be modified.  An integer  point $(x,y)\in L$  lies in $\s$ if and only if  $x^2+y^2=x^2+(1-x)^2=2x^2-2x+1=z^2$ is a perfect square. This amounts to $(2x-1)^2-2z^2=-1$, which has indeed infinitely many integer solutions (obtained from $\pm (1+\sqrt 2)^{2m+1}=2x-1+z\sqrt 2$). 
  
$\bullet$  (iii)    If $P=(0,2)$ we must have $k=0, \delta=4$ or $k=\pm 2,\delta=0$. In the first case we get the line $y=1$ of equidistant points. An integral point $(x,1)$ on it  cannot have integral distance from $O$ unless it is $(0,1)$ (we obtain the equation $1=r^2-x^2$ in integers). In the second case,   \eqref{E.sost2} yields $x=0$, so we get that $\s$ consists of the integral points on the $y$-axis.
  
$\bullet$   (iv)   If $P=(1,2)$ then $k=\pm 1, \delta=4$ and \eqref{E.sost} gives $4xy+3y^2=4x+8y-4$.  Writing this as $4x(1-y)=(y-1)(3y-5)-1$, we see that $y-1$ divides $-1$ for every solution, so $y=0,2$, and we find  $\s=\{(1,0), (0,2)\}$. (Now we have the integer points on a hyperbola with rational points at infinity: this is always a finite set.) 
  
$\bullet$  (v)   If $P=(2,2)$ then either $k=0,\delta=8$ or $k=\pm 2,\delta=4$. The first case yields $x+y=2$, and we obtain our points simply by multiplying by $2$ those in the former example $P=(1,1)$. In the second case \eqref{E.sost} becomes  $32xy-32x-32y+16=0$, which is impossible in integers.
  
$\bullet$   (vi)   Finally, if $P=(2,4)$ we must have either $k=0,\delta=20$, or $k=\pm 2,\delta=16$, or $k=\pm 4,\delta=4$. In the first case \eqref{E.sost2} gives $x+2y=5$, so by the former calculations  an integer point  $(x,y)$ is  in $\s$ if and only if $5y^2-20y+25$ is a perfect square, leading to the Pell-type equation $u^2=5(y-2)^2+5$. In turn this amounts to $u=5v$ where the integer $v$ satisfies 
  $(y-2)^2-5v^2=-1$. As in a previous  case, we obtain infinitely many integer solutions from $\pm (2+\sqrt 5)^{2m+1}=y-2+\sqrt 5 v$.
  
  The cases with $k=\pm 2$ lead to $4xy+3y^2=8x+16y-16$, i.e. $4x(2-y)=(y-2)(3y-10)-4$, hence $y-2$ divides $4$ and we obtain the points $(0,2)$ and $(4,0)$. (Of course this is as in the case $P=(1,2)$.) 
    
  The cases with $k=\pm 4$ lead to $4y(x-1)=(x-1)(3x+5)+4$, hence $x-1$ divides $4$, and we find the solutions $(-1,0), (3,4)$. (This is   similar to the previous  case, but  a nontrivial `sporadic' Pythagorean triple now appears !) 
  
  \medskip
  
  This concludes the analysis.  \end{proof}

  \section{Remarks and questions}

 1. {\tt Quantitative estimates}.  One can  ask for a quantification of Proposition \ref{P.r3}, namely   for an explicit estimate of  the cardinality $\# \s$, when $r\ge 3$, or one can ask  for an estimate of the number of points of $\s$ not collinear with the $P_i$ if these last are collinear.  For instance one can ask {\it whether there exists an absolute constant $C$ such that for collinear $P_1,P_2,P_3\in\Z^2$ there are at most $C$ points $Q\in \Z^2$, not collinear with the $P_i$ and having integral distances from each of  them}.  Already this basic question seems to escape from the known techniques. It leads to enquire about  an {\it absolute}  bound for the number of  integer points on certain curves, as e.g. the curve of genus $1$ defined by $xy(x+y)=rx+sy$ for distinct  integers $r,s\neq 0$. It is easy to derive some bound growing less that $\max(|r|,|s|)^\epsilon$ (any $\epsilon >0$) but whether an absolute bound holds, to my knowledge is  a difficult question.
  
\medskip
  
   2. {\tt Location of  points in $\s$}.  It will be noted that  all the cases when the set  `$\s-\hbox{ line} \ OP$'  is infinite in Thm. \ref{T.r2},  `come' from a hyperbola (and an associated Pell Equation), even when the points  lie on lines.  For `general' $P$, our conic $\H$ in the proof of Theorem \ref{T.r2}  is indeed a hyperbola.  In the exceptional cases $P=(1,1), (2,2), (2,4)$, when $k=0$, in fact the points of $\s$ lie on a line, but this   lifts to a hyperbola in the $xyz$-space. 
   
   As to the distribution of these points, of course their coordinates  may be explicitly expressed in terms of  linear recurrences,  have exponential growth and indeed may be parametrized as linear combinations of two exponential functions (as is well known from the theory of Pell Equation).
  
\medskip

  3.   {\tt Points of $\s$ on lines}. In case $r=2$, one can ask when there exist infinitely many points of $\s$ lying on  a single  line; indeed  (as in the previous comment) in most  cases  we     produced points in $\s$ on  a hyperbola, not a line, and the main part of Theorem \ref{T.r2} actually deals with points of $\s$ {\it outside} any finite union of lines. 
  
  By Proposition \ref{P.r3} the intersection of $\s$ with a line can be infinite  only in the trivial case of   the points $Q$   collinear with $P_1P_2$, or for  the  line  of points equidistant from $P_1,P_2$.  Concerning the latter case,  the exceptional list carries some instances of this (i.e. cases (ii), (v) and (vi)).  To give  a general recipe however  turns out  %that the issue is
   to be  quite difficult. For instance when $P:=P_2-P_1=(2r,2s)$, $r,s\in\Z$,  is such that  $r^2+s^2$ is squarefree, one can check that  integral points on the equidistant line  amount to integer solutions $(t,u)$ of the so-called  {\it negative Pell Equation}  $t^2-(r^2+s^2)u^2=-1$, widely studied.  This is known to be solvable in integers  e.g. when $r^2+s^2$ is prime, but no simple  necessary and sufficient  condition is known in general.

  \medskip
  
  4.  {\tt How many hyperbolas ?}   The proof of Thm \ref{T.r2} in most cases  exhibits a {\it single}   hyperbola containing infinitely many points in $\s$. It seems not free of interest to study {\it how many} such hyperbolas one can obtain  (and {\it which ones}) in terms of $P=(a,b)$. In other words, {\it how many irreducible %curve
   components does the Zariski-closure of the set of integral points have ?}  For instance, if $P$ has integral distance from the origin and does not lie on the axes, a well-known theorem of Fermat says that the corresponding right-angled triangle has not square area. This allows to take $k=\pm(a-b)$ in the proof (in place of $k=a\pm e$ with $e=1,2$). The general issue looks  intriguing.
  
\medskip
  
  5. {\tt Rational distances from given rational points}. The corresponding questions for {\it rational distances} from {\tt given}  {\it rational  points} in place of lattice points are sometimes (even)  easier, sometimes difficult. For instance  the set  of rational points with rational distance from two given points ($r=2$) is always infinite and actually Zariski-dense, as can be easily proved by a method similar to the above, or else using in the two equations  \eqref{E.xyz}  two parametrizations for Pythagorean triples  (we are led to rational points on a certain  rational surface).  % whose desingularization is a so-called {\it Del Pezzo surface}. 
  For distances from three given points ($r=3$) one obtains - after desingularization - an elliptic  $K3$-surface,   thus non rational. However it turns out that  this  has still a Zariski-dense set of rational points. Details for these deductions shall appear in the forthcoming note \cite{CTZ} of P. Corvaja and A. Turchet with the author.   \footnote{Elliptic $K3$-surfaces are always expected to have a dense set of rational points over some number field.}   Of course it follows from Proposition \ref{P.r3} that for $r=3$  there are only finitely many integral points (on the suitable affine part of this surface) except for trivial cases. On the other hand, it appears to be an intriguing problem to prove that for $r=3$  (or even for larger $r$)  the integral points over an {\it arbitrary} number field are never Zariski-dense.% (for totally real number fields things should be  similar to the case of $\Z$).
 \footnote{Note that  the equations that we have obtained can be considered over any number field. This  should be assumed to admit an embedding in $\R$ in case we want  to keep  the concept of `distance' used above. } 
  The issue of rational distances  for a larger number $r\ge 4$ of  given rational points is   very difficult, and related to the so-called {\it Erd\"os-Ulam problem}, and deep conjectures in Diophantine Geometry; one expects that the solutions  are never Zariski-dense: see e.g.  the exposition \cite{T} by T. Tao and the paper  \cite{ABT} by K. Ascher, L. Braune and A. Turchet.

 \vfill
 
 Umberto Zannier
 
 Scuola Normale Superiore 
 
 Piazza dei Cavalieri, 7
 
 56126 Pisa - ITALY
 
 {\it umberto.zannier@sns.it}

\end{document}